\documentclass[letterpaper, 10 pt, conference]{ieeeconf}
\pdfoutput=1
\IEEEoverridecommandlockouts
\usepackage{cite}
\usepackage{amsmath,amssymb,amsfonts}
\usepackage{graphicx}
\usepackage{textcomp}
\usepackage{graphicx}
\def\BibTeX{{\rm B\kern-.05em{\sc i\kern-.025em b}\kern-.08em
    T\kern-.1667em\lower.7ex\hbox{E}\kern-.125emX}}

\usepackage{tikz}
\usepackage{tikz-network}
\usetikzlibrary{backgrounds}
\usepackage{setspace}
\usetikzlibrary{positioning}
\usetikzlibrary{shapes.geometric}
\usepackage{ifpdf}
\usepackage[utf8]{inputenc}
\usepackage{amsmath}
\usepackage{graphicx}
\usepackage{flushend}
\usepackage{caption}
\usepackage{subcaption}
\usepackage{float}
\usepackage{comment}
\usepackage{bm}
\usepackage{xspace}
\usepackage{soul}
\usepackage{amsopn}
\usepackage{algpseudocode}
\usepackage{algorithm} 
\usepackage{enumerate}
\makeatletter
\let\NAT@parse\undefined
\makeatother
\usepackage{amsmath,color}
\usepackage[numbers,sort&compress]{natbib}
\usepackage{amsf onts}
\usepackage{color}
\usepackage{setspace}

\newcommand{\sign}{\text{sign}}
\usepackage{amssymb}
\newtheorem{theorem}{Theorem}

\newtheorem{assumption}{Assumption}

\newtheorem{lemma}{Lemma}

\usepackage{breqn}
\usepackage{hyperref} 
\usepackage{epstopdf}
\usepackage{lipsum}
\usepackage{epsfig}
\DeclareGraphicsExtensions{.pdf,.eps,.png,.jpg,.mps}

\newcommand{\R}{\mathbb{R}}

\title{\LARGE \bf
Quantized and Distributed Subgradient Optimization Method with Malicious Attack
\author{Iyanuoluwa Emiola and Chinwendu Enyioha
\thanks{The authors are with the department of Electrical and Computer Engineering,
        University of Central Florida, Orlando Florida 32816, USA
        {\tt\small iemiola@knights.ucf.edu}, {\tt\small cenyioha@ucf.edu}}%
}

}

\begin{document}

\maketitle
\thispagestyle{plain}
\pagestyle{plain}

\begin{abstract}
This paper considers a distributed optimization problem in a multi-agent system where a fraction of the agents act in an adversarial manner. Specifically, the malicious agents  steer the network of agents away from the optimal solution by sending false information to their neighbors and consume significant bandwidth in the communication process. We propose a distributed gradient-based optimization algorithm in which the non-malicious agents exchange quantized information with one another. We prove convergence of the solution to a neighborhood of the optimal solution, 
and characterize the solutions obtained under the communication-constrained environment and presence of malicious agents. Numerical simulations to illustrate the results are also presented.
 \end{abstract}

\section{Introduction}
\label{sec:introduction}
The problem of making decisions and optimizing certain objectives over a network of spatially distributed agents is one that finds application in several areas, including unmanned autonomous systems, smart grid, and increasingly distributed learning \cite{zeng2017energy,arjevani2015communication}. In such systems, each agent plays a part in solving the optimization problem by carrying out computations using available local information in coordination with neighboring agents. Certain network objectives are formulated as optimization problems of the form:
\begin{equation}\label{eqn:originall}
\min_{\bm x\in \mathcal{X}} f(\bm x) = \sum_{i=1}^n f_i(\bm x),
\end{equation}
where $n$ is the number of agents in a network, $f_i(\cdot)$ is the local objective function of agent $i$, $\bm x_i$ is the decision variable corresponding to each agent and $f(\bm x)$ is the global objective function that will be optimized. Each agent $i$ has to perform an optimization of its local objective function $f_i(\bm x_i)$ and then iteratively communicates its decision variable $x_i$ with neighbors in a network where the objective $f_i(\cdot)$ is only known by each agent $i$. 
%

Solving Problem \eqref{eqn:originall} calls for a distributed implementation in which agents carry out local computations and exchange information with neighboring agents. 
When the objective function is strongly convex, with Lipschitz continuous gradients, convergence to the optimal solution using gradient-based methods can be achieved at a linear rate \cite{magnusson2016practical,lei2016primal,nedic2009distributed}. 
It is known that noisy communication channels and adversarial nodes in the network can slow down the linear convergence rate \cite{9400255,8683442}. 
In spatially distributed systems with band-limited communication channels, where agents may resort to exchanging low-bit (quantized) information at each time-step, the resulting convergence rate is also impacted. 
Researchers have explored how quantization affects the performance of these distributed optimization algorithms \cite{4738860,1413472,7544448,9157925}.

In this paper, we assume that malicious nodes not only send bad information, but also hog the communication bandwidth. This results in the non-malicoius nodes needing to manage the limited bandwidth available by quantizing the iterates broadcast to neighboring nodes. 
To solve this, we propose a distributed subgradient with quantized and adversarial attack method (DISQAAM) where non-adversarial agents send quantized information and adversarial agents send perturbed estimates using an attack vector. Different researchers have explored how adversarial nodes and quantized information affect the performance of distributed algorithms. In \cite{8683442,8906070,7997777,ravi2019detection, li2021privacy}, the authors examined different detection mechanisms that identify malicious agents and agents that cannot be compromised by adversarial attack. In fact, the authors in \cite{ravi2019detection} even examined a detection framework for non-strongly convex functions. 

A similar problem was studied in 
 \cite{9226092,6897948,7798273}, where the convergence was obtained when information is compressed with fewer bits and in some cases to just one bit of information. Since quantization introduces errors, some authors are studying ways to compensate for the quantization errors introduced \cite{9311407}. 
 However, none of these methods simultaneously explore the constraints arising from quantization and adversarial attack and still guarantee convergence.

\subsection{Contributions}
We characterize convergence properties of a distributed subgradient algorithm in the presence of adversarial agents and limited bandwidth for communication. We show that the algorithm converges to a neighborhood of the optimal solution and that the closeness can be expressed in terms of the number of bits from the quantization and the size of the attack vector. 
    Our results show that if a step size is chosen with respect to the strong convexity and Lipschitz parameters and the subgradient bound is expressed in terms of a suitable step size, then non-adversarial agents are still able to approach the optimal solution despite the presence of adversarial agents. 
    Furthermore, the performance of the algorithm is expressed as a function of the adversarial attack vector and fineness of the quantization. 
The rest of the paper follows the following structure: In Section \ref{sec:problemformulation}, the optimization problem and attack model is presented. Section \ref{sec:Convanalysiscentralized-BB} presents the algorithm and convergence result is discussed in  Section \ref{main-result}.  Numerical experiments follow in Section \ref{sec:Numerical} to illustrate the theoretical results and the paper ends with concluding remarks in Section \ref{sec:conclude}.
%


\subsection{Notation}
We respectively denote the set of positive and negative reals as $\mathbb{R}_+$ and $\mathbb{R}_-$. We denote a vector or matrix transpose as $(\cdot)^T$, and the L$2$-norm of a vector by $||\cdot||$. We also denote the gradient of a function $f(\cdot)$ as $\nabla f(\cdot)$ and an $n$ dimensional vector of ones as $1_n$.

\section{Problem formulation }\label{sec:problemformulation}
Suppose we have an undirected graph $G = (\mathcal{V},\mathcal{E})$ comprising $n$ nodes where $\mathcal{V} = {1,2, ... n}$ is the set of nodes (agents) and   $\mathcal{E}={(i,j)}$ is the set of edges. Let the neighbors of each agent $i$ be denoted by the set $N_i = \{j: \ (i,j) \in \mathcal{E}\}$. 
%
The agents are to jointly solve the problem
\begin{equation}\label{eqn:quadratic}
\min_{\bm x\in \mathcal{X}} f(\bm x) = \sum_{i=1}^n f_i(\bm x),
\end{equation}
where each agent $i$ has a component $f_i(\cdot)$ of the objective function. We assume that some agents in the network act in an adversarial (malicious) manner by perturbing their estimate at each iteration and overwhelming the communication bandwidth for coordination in the network. To manage the limited bandwidth left, the non-adversarial nodes quantize the information shared with neighboring nodes. 
The uniform quantizer is chosen to ensure that agents use equal and constant step sizes to broadcast information to their neighbors. 
We assume the local objective function $f_i(\cdot)$ in equation \eqref{eqn:quadratic} is strongly convex and $\mathcal{X}$ is the feasible set. 
%
\def\nodesize{0.5}
\def\thinpipe{1.5pt}
\def\largepipe{3pt}
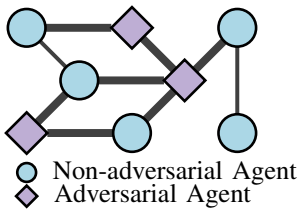
\begin{figure}[h]
\begin{center}
\begin{tikzpicture}[scale=0.7]
    \begin{scope}[shift={(0,2)},rotate=-90]
            \Vertex[size=\nodesize,position=270,distance=1mm]{A}
        \Vertex[x=1,y=1,size=\nodesize]{B}
        \Vertex[x=2,size=1.1*\nodesize,RGB,color={190,174,212},shape=diamond]{C}
        \Vertex[y=2,size=1.1*\nodesize,RGB,color={190,174,212},shape=diamond]{D}
        \Vertex[x=2,y=2,size=\nodesize]{E}
        \Vertex[x=1,y=3,size=1.1*\nodesize,RGB,color={190,174,212},shape = diamond]{F}
        \Vertex[y=4,size=\nodesize]{G}
        \Vertex[x=2,y=4,size=\nodesize]{H}
        
        \Edge[lw=\thinpipe](A)(B)
        \Edge[lw=\thinpipe](G)(H)
        \Edge[lw=\largepipe](A)(D)
        \Edge[lw=\largepipe](B)(C)
        \Edge[lw=\largepipe](B)(F)
        \Edge[lw=\largepipe](C)(E)
        \Edge[lw=\largepipe](D)(F)
        \Edge[lw=\largepipe](E)(F)
        \Edge[lw=\largepipe](F)(G)
    \end{scope}
    
    \Vertex[x=0,y=-0.75,size=0.5*\nodesize,label={\normalsize Non-adversarial Agent},position=0,distance=1mm]{Z}
    \Vertex[x=0,y=-1.2,size=0.5*\nodesize,RGB,color={190,174,212},shape = diamond,label={\normalsize Adversarial Agent},position=0,distance=1mm]{X}
\end{tikzpicture}
\end{center}
 \caption{Adversarial behavior in WSN }
\label{fig:adversarial-bandwidth}
\end{figure}

As illustrated in Figure \ref{fig:adversarial-bandwidth}, the thin communication links are used to denote connection between two non-adversarial agents, while thick pipes are used to depict the connection between an adversarial agent and any other agents (adversarial or non-adversarial). The non-adversarial nodes need to manage the communication bandwidth to be able to approach the optimal solution to equation \eqref{eqn:quadratic} to overcome the bottleneck caused by adversarial nodes upscaling their estimates.
We solve problem \eqref{eqn:quadratic} using the distributed subgradient method. In this framework, each non-adversarial agent $i$ updates a local copy $\bm x_i\in \mathbb{R}^p$ of the decision variable $\bm x\in \mathbb{R}^p$ to send quantized iterate $Q(\bm x_i(k))\in \mathbb{R}^p$ and carry out a local update according to:
\begin{equation}\label{eqn:gradient-quantized}
    \bm x_i(k+1) = \bm x_i(k) - \bm q_i(k)+\sum_{j\in N_i\cup\{i\}} w_{ij}\bm q_j(k)-\alpha_i \nabla f_i(\bm x(k)),
\end{equation}
where $\alpha_{i} \in \mathbb{R}_+$ is the step size, $w$ is the weight matrix, and $q_i(k)=Q^{i}_k(\bm x_i(k))$ is the quantized value of $\bm x_i(k)$. 
\subsection{The Uniform Quantizer}
  Let $\bm{ x} \in \R^d$.
A uniform quantizer with step size $\Delta$ and mid-value $\bm{x^\prime}$ is
$ \bm Q(x) = \bm{x^\prime} + \sign(\bm x-\bm{x^\prime}) .\Delta . \lfloor \frac{\|\bm x - \bm{x^\prime} \|}{\Delta} + \frac{1}{2} \rfloor$,
where $\Delta = \frac{\ell}{2^b}$, $\ell$ is the size of the quantization interval and $b$ is the number of bits. $\lfloor \cdot \rfloor$ is the floor function and $\sign(\bm x)$ is the sign function.
Let $b$ the number of bits, the quantization interval be set to $[\bm{x^\prime}-1/2, \bm{x^\prime}+1/2]$ the uniform quantizer be denoted as $\bm Q_k^i(\bm x_i(k))$ and mid-value $\bm x^{\prime^k}_{\Delta,i}$.
Let $\bm{q}_i(k) = \bm Q_k^i(\bm x_i(k))$; then, the quantization error, $\Delta_i(k) = \bm q_i(k) - \bm x_i(k)$. Suppose $\ell_i$ be the quantization interval size for each agent $i$, the quantization error bound of a uniform quantizer is given by 
    $\|\Delta_i(k)\| \le \frac{ \bm \ell_i}{2^{b+1}}$.
\subsection{Attack Model}
The aim of the malicious agents is to prevent the network from reaching the optimal solution to Problem \eqref{eqn:quadratic}, by perturbing their estimates with either a positive or negative attack vector $e_i(k) \in \mathbb{R}^p$ (with all entries of the vector being positive or negative) according to the following iteration:
\begin{align}\label{eqn:gradientmalicious}
  \bm x_i(k+1) &= \bm x_i(k) - \bm q_i(k) \\ 
  \quad &+\sum_{j\in N_i\cup\{i\}} w_{ij}\bm q_j(k)-\alpha_i \nabla f_i(\bm x(k)) {+} e_i(k). \nonumber
\end{align}
After the step in equation \eqref{eqn:gradientmalicious}, the adversarial agents broadcast their estimates to their neighbors.
We note that the adversarial agents can choose either the same attack vector $e(k)$ obtained by taking an average of estimates in a complete graphical structure or different attack vector $e_i(k)$ at every iteration in a general graphical structures (where each agent $i$ chooses an attack vector). We analyze the general graphical structures scenario and details regarding these graphical structures in an adversarial case are seen in \cite{9400255}.
The following are assumed on Problem \eqref{eqn:quadratic}:

\begin{assumption}\label{Assumption1}
The cost function $f(\bm x)$ in Problems \eqref{eqn:quadratic} is strongly convex.
This implies that for any vectors $x, y \in \mathbb{R}^{p}$, there exists a strong convexity parameter $\mu\in\mathbb{R}_+$, with $\mu \leq L$ (where $L$ is the Lipschitz constant) such that:
\begin{equation*}
    f(\bm x)\geq f(\bm y)+\nabla f(\bm y)^{T}(\bm x-\bm y)+\frac{\mu}{2}\|\bm x-\bm y\|^2.
\end{equation*}
\end{assumption}
\begin{assumption}\label{Assumption5}
The subgradient $g_i$ of $f_i$ at $\bm x_i$ is uniformly bounded by $\bar L_i$ in the feasible set $\mathcal{X}$. This implies that there exists $\bar L_i>0$ such that $\| g_i(\bm x_i)\|\leq \bar L_i$,
where for all $y$, the relationship $f_i(y)\geq f_i(x)+g_i^{T}(y_i-x_i)$ holds.
\end{assumption}
Using the above assumptions, we show below that convergence is attained despite the two constraints of malicious attack and quantization described.  
\section{Distributed Subgradient Convergence Analysis with Quantization and Attack}\label{sec:Convanalysiscentralized-BB}
We now proceed to the convergence analysis of the update Equation \eqref{eqn:gradientmalicious} used to solve Problem \eqref{eqn:quadratic}. The goal is to analyze convergence to the optimal solution of the minimization problem in equation \eqref{eqn:quadratic} in the presence of quantization and attack as described in section \ref{sec:problemformulation}. 
In equation \eqref{eqn:gradient-quantized}, each non-adversarial and adversarial agent $i$ achieves consensus with other nodes in the network by taking an average of his estimates and those of his neighbors. This averaged consensus includes non-adversarial and adversarial agents' estimates.
%
Let  $X=[\bm x_1; \ \bm x_2; \ \hdots \ \bm x_n]^T \in \mathbb{R}^{np}$
be the concatenation of the local variables $x_i$, $I_p$ be the identity matrix with dimension $p$ and let $\bm\Xi=[\bm \xi_1; \ \bm \xi_2; \ \hdots \ \bm \xi_n]^T \in \mathbb{R}^{np}$ be the concatenation  of $\bm \xi$. $H$ denotes the concatenation of the local variables of $h_i$.
When quantization occurs among agents during broadcasting of information, the quantized values can have solutions that are not feasible when subjected to constraints. In this regard, we account for the error due to projection unto the feasible set. Let $\bm{h} \in \R^d$, and $\bm{\xi}(\bm{h})$ be the error based on projection of $\bm{h}$ in the feasible set $\cal{X}$. By definition based on the description above,:
$  \bm{\xi}(\bm{h}) = \bm{h} - [\bm{h}]_{\cal{X}}$.
Another representation of Equation~\eqref{eqn:gradient-quantized} is given by:
\begin{align}
    \bm{h}_i(k) = & \sum_{j\in\cal{N}_{i}} w_{ij} \bm x_j(k) + \bm x_i(k) - \bm q_i(k) \nonumber\\
         & + \sum_{j\in \cal{N}_i} w_{ij}(\bm q_j(k)-\bm x_j(k)) - \alpha(k) \bm g_i(x_i(k)).
\end{align}
The iterative update equation is now given as:
\begin{align}
    \bm x_i(k+1) = [\bm{h}_i(k)]_{\cal{X}} = \bm{h}_i(k) - \bm{\xi}_i(\bm{h}_i(k)).
\end{align}
We obtain the matrix form of the above update as
\begin{align}
    \bm{H}(k) = \bm{W} \bm{X}(k) {+} (\bm{I}{-}\bm{W})(\bm{X}(k){-}\bm{Q}(k)) {-} \alpha(k) \bm{G}(\bm{X}(k)),
\end{align}
\begin{align}
    \bm{X}(k+1) = \bm{H}(k) - \bm{\Xi}(\bm{H}(k)),
\end{align}
and $\bm{W}$ is the doubly stochastic matrix.
Suppose $\bm{\bar{x}}(k)$ and $\bm{\bar{\xi}}(k)$ are the mean of $\bm x_i(k)$ and $\bm \xi_i(h_i(k))$ respectively, we obtain
$ \bm{\bar{x}}(k) = \frac{1}{n} \sum_{i=1}^n \bm x_i(k) = \frac{1}{n} \bm{X}^T \bm{1} \in \R^d$, and
    \begin{equation}\label{projection-error-formula}
    \bm{\bar{\xi}}(k) = \frac{1}{n} \sum_{i=1}^n \bm \xi_i(h_i(k)) = \frac{1}{n}\left(\bm\Xi(k)\right)^T \bm{1} \in \R^d.
\end{equation}
We can define the quantization error for each agent i as $\Delta_i(k)=x_i(k)-q_i(k)$ and the average of the errors as $\Delta(k) = \displaystyle\frac{1}{n} \sum_{i=1}^n \Delta_i(k)$.
We now have obtain:
$\bm{\bar h}(k) = \bm{\bar x}(k) - \frac{\alpha(k)}{n} \sum_{i=1}^n \bm g_i(\bm x_i(k))$,
and $\bm{\bar x}(k+1) = \bm{\bar h}(k) - \bm{\bar \xi}(k)$. Thus, we obtain the iterative equation:
\begin{align} \label{eq:bar_x}
    \bm{\bar x}(k+1) =  \bm{\bar x}(k) - \frac{\alpha(k)}{n} \sum_{i=1}^n \bm g_i(\bm x_i(k)) - \bm{\bar \xi}(k).
\end{align}
Now we introduce a Lemma that accounts for the bounds of the projection error according to equation \eqref{projection-error-formula}.
\begin{lemma}\label{error-projection}

Let Assumptions \ref{Assumption1} and \ref{Assumption5} hold with $\Delta(k)$ being the average of the quantization errors. If the step size $\alpha$ for each agent $i$ is constant such that $\alpha\leq 1$, the error due to projection is bounded given by the following relationship:
 \[
\|\bm{\bar\xi}(k)\| \le \sqrt{8} \Delta(k) + \sqrt{2} \frac{\bar L}{n}\alpha.
 \]
\end{lemma}
\begin{proof}
See Appendix \ref{error-projection-proof}.
\end{proof}
\section{Main Result}\label{main-result}
This section explores the convergence analysis according to the proposition in sections \ref{sec:problemformulation} and \ref{sec:Convanalysiscentralized-BB}. This is shown in the following theorem.
\begin{theorem}
\label{distributed-quantized-malicious}
 Let Assumptions \ref{Assumption1} and \ref{Assumption5} hold, 
 and the step size $\alpha$ satisfies
 $\alpha \leq \frac{2}{\mu + L} \quad \text{and} \quad \frac{\mu + L}{3\mu L} \leq \alpha \leq \frac{\mu + L}{2\mu L}$.
 Given that the size of a uniform quantization interval with $b$ bits be upper-bounded by
 $\bm \ell \le  \frac{2^b}{ \sqrt{6}}$,
and the subgradient bound be upper-bounded by $\bar L \le 1/\sqrt{6}\alpha$,
then the iterates generated when non-adversarial send quantized estimates converge to a neighborhood of the optimal solution, $x^*$ with the neighborhood size given by  
$\frac{\sqrt{6}(l+2^b \bar L\alpha)+2^b \sqrt{3}\|\bm e(k)\|}{2^b}$.

\end{theorem}
\begin{proof}
Since $\bm x^*$ is the optimal solution of \eqref{eq:bar_x} and $\bm x^a$ is the adversary according to $\bm x^a = \bm x^* + \bm e(k)$, we obtain:
\begin{align*}
    \| &\bm{\bar x}(k+1) - \bm x^* - \bm e(k) \|^2 \\
        &= \| \bm{\bar x}(k) - \bm x^* - \bm e(k)  - \frac{\alpha(k)}{n} \sum_{i=1}^n \bm g_i(\bm x_i(k)) -\bm{\bar\xi}(k) \|^2.
\end{align*}
By expansion, we obtain the following:
\begin{align*}
    \| &\bm{\bar x}(k+1) {-} \bm x^* {-} \bm e(k) \|^2 \\
        &= \| \bm{\bar x}(k+1) {-} \bm x^* \|^2 {+} \|\bm e(k)\|^2 {+} \left\|\frac{\alpha(k)}{n} \sum_{i=1}^n \bm g_i(\bm x_i(k)) \right\|^2\\
        &\quad {+} \|\bm{\bar\xi}(k)\|^2 {+} 2 \bm e(k) \left( \frac{\alpha(k)}{n} \sum_{i=1}^n \bm g_i(\bm x_i(k))  {-} (\bm{\bar x}(k) {-} \bm x^*) \right)\\
        &\quad {-} 2 (\bm{\bar x}(k) {-} \bm x^*)^T \left( \frac{\alpha(k)}{n} \sum_{i=1}^n \bm g_i(\bm x_i(k)) \right)\\
        &\quad {+} 2 \bm{\bar\xi}(k) \left( \frac{\alpha(k)}{n} \sum_{i=1}^n \bm g_i(\bm x_i(k)) {-} (\bm{\bar x}(k) - \bm x^*) \right)
       {+} 2 \bm e(k) \bm{\bar\xi(k)}.
\end{align*}
By inspection, the upper bound of the preceding expression $\| \bm{\bar x}(k+1) - \bm x^* - \bm e(k) \|^2$ has eight terms. In what follows, we bound some of the eight terms, starting with the fifth term. 

\begin{align*}
    & 2 \bm e(k) \left( \frac{\alpha(k)}{n} \sum_{i=1}^n \bm g_i(\bm x_i(k))  - (\bm{\bar x}(k) - \bm x^*) \right)\\
    & \quad \le  \|\bm e(k)\|^2  + \left\|\frac{\alpha(k)}{n} \sum_{i=1}^n \bm g_i(\bm x_i(k)) - (\bm{\bar x}(k) - \bm x^*)\right\|^2\\
        & \qquad - 2(\bm{\bar x}(k) - \bm x^*)^T \left( \frac{\alpha(k)}{n} \sum_{i=1}^n \bm g_i(\bm x_i(k)) \right),\\
    & \quad \le \|\bm e(k)\|^2 + \frac{\alpha^2}{n^2} \left\|\sum_{i=1}^n \bm g_i(\bm x_i(k))\right\|^2 + \| \bm{\bar x}(k) - \bm x^* \|^2\\
        & \qquad - \frac{\alpha}{n} c_1 \left\|\sum_{i=1}^n \bm g_i(\bm x_i(k))\right\|^2 - \alpha c_2 \| \bm{\bar x}(k) - \bm x^* \|^2,
\end{align*}
where $c_1 = \dfrac{2}{\mu +L}$ and $c_2 = \dfrac{2 \mu L}{\mu +L}$.
We proceed by bounding the second term of the derived upper bound of $2 \bm e(k) \left( \frac{\alpha(k)}{n} \sum_{i=1}^n \bm g_i(\bm x_i(k))  - (\bm{\bar x}(k) - \bm x^*) \right)$. In this regard, we have the following bound:
\begin{equation}\label{sum-gradient-bound}
    \left\|\sum_{i=1}^n \bm g_i(\bm x_i(k))\right\| \le \sum_{i=1}^n\| g_i(\bm x_i(k))\|.
\end{equation}
By squaring both sides of the equation \eqref{sum-gradient-bound}, we obtain:
$ \|\sum_{i=1}^n \bm g_i(\bm x_i(k))\|^2 \le (\sum_{i=1}^n\|g_i(\bm x_i(k))\|)^2$. Therefore we have the relationship:
\begin{align*}
&2 \bm e(k) \left( \frac{\alpha(k)}{n} \sum_{i=1}^n \bm g_i(\bm x_i(k)) {-} (\bm{\bar x}(k) {-} \bm x^*) \right)\\
& \qquad \le \|\bm e(k)\|^2 {+} \frac{\alpha^2}{n^2} \left(\sum_{i=1}^n \bm \|g_i(\bm x_i(k))\|\right)^2 {+} \| \bm{\bar x}(k) - \bm x^* \|^2\\
& \qquad~ {-} \frac{\alpha}{n} c_1 \left(\sum_{i=1}^n \bm \|g_i(\bm x_i(k))\|\right)^2 {-} \alpha c_2 \| \bm{\bar x}(k) {-} \bm x^* \|^2.
\end{align*}
We recall that the expression $\| \bm{\bar x}(k+1) - \bm x^* - \bm e(k) \|^2$ is upper-bounded by eight terms. Now we bound the seventh term, $2 \bm{\bar\xi}(k) ( \frac{\alpha(k)}{n} \sum_{i=1}^n \bm g_i(\bm x_i(k)) - (\bm{\bar x}(k) - \bm x^*) )$ to obtain:
\begin{align*}
    & 2 \bm{\bar\xi}(k) \left( \frac{\alpha(k)}{n} \sum_{i=1}^n \bm g_i(\bm x_i(k)) - (\bm{\bar x}(k) - \bm x^*) \right)\\
        & \qquad \le \|\bm{\bar\xi}(k)\|^2 + \frac{\alpha^2}{n^2} \left(\sum_{i=1}^n \bm \|g_i(\bm x_i(k))\|\right)^2 + \| \bm{\bar x}(k) - \bm x^* \|^2\\
        & \qquad~~ - \frac{\alpha}{n} c_1 \left(\sum_{i=1}^n \bm \|g_i(\bm x_i(k))\|\right)^2 - \alpha c_2 \| \bm{\bar x}(k) - \bm x^* \|^2,
\end{align*}\\
By bounding the expression as $2 \bm e \bm{\bar\xi}(k) \le \|\bm e(k)\|^2 + \|\bm{\bar\xi}(k)\|^2$, we obtain the combination of all bounds obtained as:
\begin{align*}
    \|&\bm{\bar x}(k+1) - \bm x^* -\bm e(k)\|^2 \\
        & \le \|\bm{\bar x}(k) {-} \bm x^*\|^2 {+} \|\bm e(k)\|^2 {+} \frac{\alpha^2}{n^2} (\sum_{i=1}^n \bm \|g_i(\bm x_i(k))\|)^2 {+}\|\bm{\bar\xi}(k)\|^2\\
            & \quad {+} \|\bm e(k)\|^2 {+} \frac{\alpha^2}{n^2} \left(\sum_{i=1}^n \bm \|g_i(\bm x_i(k))\|\right)^2 {+} \|\bm{\bar x}(k) {-} \bm x^*\|^2\\
            & \quad - \frac{\alpha}{n} c_1 \left(\sum_{i=1}^n \bm \|g_i(\bm x_i(k))\|\right)^2 - \alpha c_2 \| \bm{\bar x}(k) - \bm x^* \|^2\\
            & \quad - \frac{\alpha}{n} c_1 \left(\sum_{i=1}^n \bm \|g_i(\bm x_i(k))\|\right)^2 - \alpha c_2 \| \bm{\bar x}(k) - \bm x^* \|^2 \\
            & \quad + \|\bm{\bar\xi}(k)\|^2 + \frac{\alpha^2}{n^2} \left(\sum_{i=1}^n \bm \|g_i(\bm x_i(k))\|\right)^2 + \|\bm{\bar x}(k) - \bm x^*\|^2\\
            & \quad - \frac{\alpha}{n} c_1 \left(\sum_{i=1}^n \bm \|g_i(\bm x_i(k))\|\right)^2 - \alpha c_2 \| \bm{\bar x}(k) - \bm x^* \|^2\\
            & \quad + \|\bm e_i(k)\|^2 + \|\bm{\bar\xi}(k)\|^2,\\
        & = (3 - 3 \alpha c_2) \| \bm{\bar x}(k) - \bm x^* \|^2 + 3 \|\bm e(k)\|^2 + 3 \|\bm{\bar\xi(k)}\|^2\\
            & \quad + \left( \frac{3\alpha^2}{n^2} - \frac{3\alpha}{n} c_1 \right) \left(\sum_{i=1}^n\|g_i(\bm x_i(k))\|\right)^2.
\end{align*}
We will show that
$\left( \frac{3\alpha^2}{n^2} - \frac{3\alpha}{n} c_1 \right) \le 0$,
when $\alpha \le c_1$. To do this, it suffices to show that
$3\alpha^2 - 3\alpha n c_1 \le 0$.
We know that the root of the equation in $\alpha$ of ${\alpha(3\alpha - 3 n c_1) = 0}$ is $\alpha = 0$ and $\alpha = n c_1$, and we have solution $ 0 \le \alpha \le n c_1$. So $\alpha \in [0, n c_1]$.
If $\alpha \le c_1$ and $n \ge 1$, it implies that $\alpha \le n c_1$.
Alternatively, if $\alpha \le c_1$, then $\alpha^2 \le \alpha c_1$ and consequently $ \frac{3\alpha^2}{n^2} - \frac{3 \alpha n c_1}{n} \le 0$
for $n> 0$. We now obtain
$\left( \frac{3\alpha^2}{n^2} - \frac{3 \alpha n c_1}{n} \right) \left( \sum_{i=1}^n \|\bm g_i(\bm x_i(k))\| \right)^2 \le 0$.
Therefore the bounds on $\| \bm{\bar x}(k+1) - \bm x^* - \bm e_i(k) \|^2$ is:
 \begin{align*}
     \| \bm{\bar x}(k+1) - \bm x^* - \bm e(k) \|^2 & \le (3 - 3\alpha c_2)  \|\bm{\bar x}(k) - \bm x^*\|^2\\
        & \quad + 3\|\bm e(k)\|^2 + 3 \|\bm{\bar\xi(k)}\|^2.
\end{align*}
To show that $3 - 3 \alpha c_2 \ge 0$, we show that $3 \alpha c_2 \le 3 \Rightarrow \alpha c_2 \le 1$.
If $\alpha \le 1/\mu$, then we have the following:
$\alpha c_2 \le \frac{1}{\mu} \frac{2\mu L}{\mu+L} = \frac{2L}{\mu+L}$.
If $\mu = L$, then $\mu+L = 2L \Rightarrow \alpha c_1 \le 1$.
Since $\alpha c_2 \le 1$, then we have that $(3 - 3\alpha c_2) \ge 0$.
For $3-3\alpha c_2$ not to grow unbounded we need ${(3-3\alpha c_2)\in(0,1)}$. This implies that we need $3(1-\alpha c_2) \le 1$ or equivalently when $\alpha \ge \frac{2}{3 c_2} =\frac{\mu+L}{3\mu L}$.
If $\alpha \in \left(\frac{\mu+L}{3\mu L} , \frac{1}{\mu} \right)$, then the expression $(3 - 3\alpha c_2)$ will not grow unbounded.
From Lemma \ref{error-projection}, we obtain:
    $\|\bm{\bar\xi}(k)\|^2 \le 8 \|\Delta(k)\|^2 + 2 \bar L^2\alpha^2$.
and it leads to:
\begin{align*}
    &\|\bm{\bar x}(k+1) - \bm x^*-\bm e(k)\|^2 \\
    &\quad\le (3-3\alpha c_2) \|\bm{\bar x}(k) - \bm x^*\|^2 + 3\|\bm e(k)\|^2 + 24\|\Delta(k)\|^2 + 6\bar L^2\alpha^2.
\end{align*}
This leads to following relationship:
\begin{align*}
    &\|\bm{\bar x}(k+1) - \bm x^*-\bm e(k)\|^2 \\
    &\quad\le (3-3\alpha c_2) \|\bm{\bar x}(k) - \bm x^*\|^2 + 3\|\bm e(k)\|^2 + \frac{24(\bm \ell)^2}{2^{2b+2}} + 6\bar L^2\alpha^2,\\
    &\quad\le (3-3\alpha c_2) \|\bm{\bar x}(k) - \bm x^*\|^2 + 3\|\bm e(k)\|^2 + \frac{6(\bm \ell)^2}{2^{2b}} + 6\bar L^2\alpha^2.
\end{align*}
If $\bm{\bar x}(k+1) - \bm x^* < 0 < \bm e(k)$, then $\|\bm{\bar x}(k+1) {-} \bm x^*{-}\bm e(k)\| \ge \|\bm{\bar x}(k{+}1) - \bm x^*\|$. In addition, if $e(k)\leq 0$, then $\|\bm{\bar x}(k+1) - \bm x^*-\bm e(k)\| \ge \|\bm{\bar x}(k+1) - \bm x^*\|$ also holds. Therefore, the bound on $\|\bm{\bar x}(k+1) {-} \bm x^*{-}\bm e(k)\|^2$ is
\begin{align*}
    &\|\bm{\bar x}(k+1) {-} \bm x^*\|^2 \\
    &\quad\le (3{-}3\alpha c_2) \|\bm{\bar x}(k) {-} \bm x^*\|^2  {+} 3\|\bm e(k)\|^2 {+} \frac{6}{2^{2b}}(\bm \ell)^2 {+} 6\bar L^2\alpha^2.
\end{align*}
By applying recursion principles, we obtain the following:
\begin{align*}
    &\|\bm{\bar x}(k) - \bm x^*\|^2 \\
        &~\le (3{-}3\alpha c_2)^k \|\bm{\bar x}(0) {-} \bm x^*\|^2 {+} 3\|\bm e(k)\|^2 {+} \frac{6}{2^{2b}}(\bm \ell)^2 {+} 6\bar L^2\alpha^2,
\end{align*}
Equivalently, we the following relationship:
\begin{align}
    \|\bm{\bar x}(k) {-} \bm x^*\| & \le (3{-}3\alpha c_2)^{k/2} \|\bm{\bar x}(0) {-} \bm x^*\| {+} \sqrt{3}\|\bm e(k)\| \nonumber \\ 
    & \quad {+} \sqrt{\frac{6}{2^{2b}}}\bm \ell + \sqrt{6} \bar L\alpha \label{upper-bound}
\end{align}
For $\sqrt{6/2^{2b}}\bm \ell$ to be small, we need $\sqrt{6/2^{2b}}\bm \ell \leq 1$, which implies that $\bm \ell \le \frac{1}{\sqrt{\frac{6}{2^{2b}}}} = \frac{2^b}{ \sqrt{6}}$.
In addition, if $\sqrt{6}\bar L\alpha \le 1$ or $\bar L \le 1/\sqrt{6}\alpha$, the size of the neighborhood is given by:
$\frac{\sqrt{6}(l+2^b L\alpha)+2^b \sqrt{3}\|\bm e(k)\|}{2^b}$,
and the non-adversarial agents can converge to the neighborhood of the optimal solution, $\bm x^*$.
 
There are trade-offs in the results obtained by combining adversarial attack and quantization as constraints. This is evident in the proof of Theorem \ref{distributed-quantized-malicious}, where we needed the condition $\|\overline{x}(0)-x^*\| <\|e(k)\|$ when $e(k)>0$ for the algorithm to converge. However, such condition is not needed when $e(k)\leq 0$. In addition, due to strong convexity assumption and boundedness on the coefficient of the error term, the step size lies in the domain  $(\frac{\mu + L}{3\mu L}, \frac{\mu + L}{2\mu L})$. Moreover, the subgradient bound needs to depend on the step size as well to aid convergence. 
 
 In addition, our previous result \cite{9400255} show that increasing the number of adversarial agents leads to an increase in the convergence neighborhood. However, when quantization is added as a constraint, increasing the number of bits leads to reduction of the error bounds. While the authors in \cite{4738860,7544448,9157925} show that convergence of distributed subgradient mthods with qualtization depends on the quantization levels and the number of bits, our proposed method adds adversarial attack to the constraint and still obtain similar convergence attributes.

\end{proof}

\section{Numerical Experiments}\label{sec:Numerical}


We illustrate our theoretical results via simulations by considering a network comprising $n= 10$ agents with a fraction of $n$ acting in an adversarial manner. At each time step, the malicious agents use the algorithm in Equation \eqref{eqn:gradientmalicious}, with each malicious agent using a different attack vector.  
The strongly convex objective function to solve the problem in equation \eqref{eqn:quadratic} used for simulation is given by the following:  
\begin{equation}\label{eqn:simulationquadratic}
    f(x)=\frac{1}{2}x^Tx
\end{equation}
We examine the influence of the uniform quantizer used by the non-adversarial agents by varying the number of bits in order to examine the convergence to optimal solution or using the error. We consider attack vectors sent by adversarial nodes in conjunction with the uniform quantizer and observe that convergence to the neighborhood of the optimal solution is still achieved. We use a step size of $ \alpha = 0.7$ and attack vector entries in the interval $(0,1)$ as test values.

\begin{figure*}[!t]
     \centering
     \begin{subfigure}[t]{0.31\textwidth}
         \centering
          \includegraphics[width=1\linewidth]{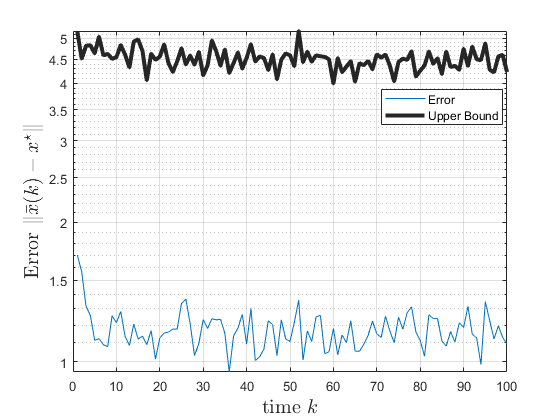}
    \caption{$7$ non-adversarial agents with $1$ bit}
    \label{fig:iteration-7-non1bit-diff}
     \end{subfigure}
     \hfill
     \begin{subfigure}[t]{0.31\textwidth}
         \centering
           \includegraphics[width=1\linewidth]{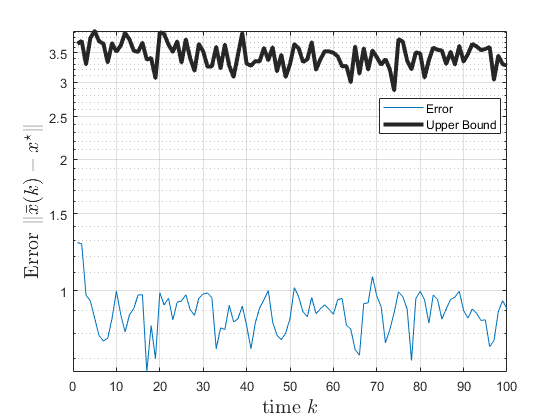}
    \caption{$7$ non-adversarial agents with $5$ bits.}
    \label{fig:iteration-7-non5bits-diff}
     \end{subfigure}
    \hfill
     \begin{subfigure}[t]{0.31\textwidth}
         \centering
        \includegraphics[width=1\linewidth]{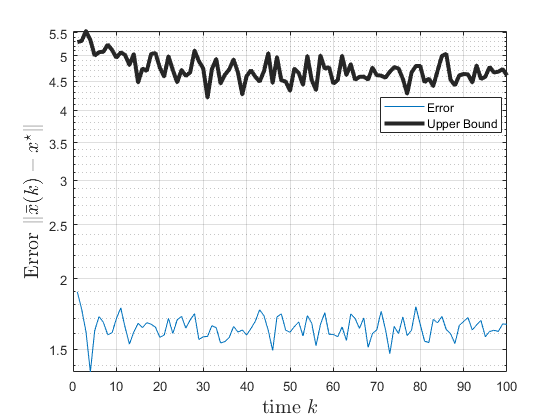}
    \caption{$3$ non-adversarial agents with $1$ bit}
    \label{fig:iteration-3-non1bit-diff}
     \end{subfigure}
        \caption{Quantization and Adversarial Simulation}
        \label{fig:Quantization and Adversarial Simulation}
\end{figure*}
To obtain more insight on the claims made in Section  \ref{sec:Convanalysiscentralized-BB}, we illustrate two scenarios ($1$ and $5$ bits uniform quantizer) where non-adversarial agents constitute $70\%$ of the total number of agents in the network and a scenario where non-adversarial agents account for $30\%$ of the total number of agents in the network. We note that the adversarial agents send different attack vectors in the interval $(0,1)$. As seen in figures \ref{fig:iteration-7-non1bit-diff}, \ref{fig:iteration-7-non5bits-diff} and \ref{fig:iteration-3-non1bit-diff}, non-adversarial agents are still able to approach the neighborhood of the optimal solution despite quantization and adversarial attack. When there are less adversarial agents in the network, we obtain a larger error with $1$ bit than with $5$ bits as seen in figures \ref{fig:iteration-7-non1bit-diff} and \ref{fig:iteration-7-non5bits-diff}. When there are more adversarial agents in the network as seen in figure \ref{fig:iteration-3-non1bit-diff}, we obtain a larger error than in figures \ref{fig:iteration-7-non1bit-diff} and \ref{fig:iteration-7-non5bits-diff} due to the constraints arising from a network congested with adversarial agents and using $1$ bit of information.
Finally, from Figure \ref{fig:Quantization and Adversarial Simulation}, we observed that the obtained iterates and theoretical bound are reasonably close. 

\section{Conclusion}
\label{sec:conclude}
We explore the performance of a distributed subgradient algorithm with two constraints described. By using a strongly convex function, an adequate step size and a uniform quantizer, we show convergence to the neighborhood of the optimal solution of the distributed optimization problem considered. We explore scenarios where adversarial agents upscale their estimates using different vector at each iteration. We also show that convergence to a neighborhood of the optimal solution is unaffected  even with one bit of information being exchanged by non-adversarial agents in a network. Numerical experiments affirm that when adversarial agents send an attack vector at each iteration, increasing the number of bits mostly leads to reduction of errors in approaching the optimal solution despite the presence of attack values. However, as the number of adversarial agents increase in the network, the convergence error increases. Our result holds for strongly convex functions under certain conditions and interested researchers should consider the problem described in this paper for non-strongly convex functions as an open problem.

\section{Appendix}
\subsection{Proof of Lemma 1}\label{error-projection-proof}
\begin{proof}
We begin with the relationship:
\begin{equation}\label{error-average}
    \bm{\bar\xi}(\bm
    (k) = \frac{1}{n}\sum_{i=1}^n \bm{\xi}_i(\bm
   h_i (k)).
   \end{equation}
   By squaring both sides of equation \eqref{error-average}, 
and using the bound of $\sum_{i=1}^n \left\|\bm{\xi}_i(\bm
    h_i(k))\right\|^2$ shown in \cite{9157925} as:
\begin{align}
   \sum_{i=1}^n \left\|\bm{\xi}_i(\bm
    h_i(k))\right\|^2 \le 8 \sum_{i=1}^n \|\Delta_i(k)\|^2 + 2 \bar L^2 \alpha^2(k)
\end{align}
we obtain 
 $\|\bm{\bar\xi}(\bm
    (k)\|^2 \le \frac{8}{n^2} \sum_{i=1}^n \left\|\Delta_i(k)\right\|^2 + \frac{2 \bar L^2 \alpha^2}{n^2}$. Equivalently, we have:
\begin{equation}\label{projection-norm-error-bound}
   \|\bm{\bar\xi}(\bm
    (k)\| \le \frac{\sqrt{8}}{n} \left(\sum_{i=1}^n \left\|\Delta_i(k)\right\|^2\right)^{1/2} + \frac{\sqrt{2} \bar L \alpha}{n}.
\end{equation}
If $\ell \le 1$, then $\|\Delta_i(k)\| \le 1$ and $\|\Delta_i(k)\|^2 \le 1$. So it yields
$\sum_{i=1}^n \|\Delta_i(k)\|^2 \le \sum_{i=1}^n 1 = n$.
From equation \eqref{projection-norm-error-bound}, 
\begin{equation*}
     \frac{\sqrt{8}}{n} \left(\sum_{i=1}^n \|\Delta_i(k)\|^2\right)^{1/2} \le \frac{\sqrt{8}\sqrt{n}}{n} = \frac{\sqrt{8}}{\sqrt{n}} = \sqrt{\frac{8}{n}}.
\end{equation*}
When $n\ge 1$, 
we obtain the bound:
\begin{equation}\label{error-sum-bound}
    \sqrt{8} \left(\sum_{i=1}^n \|\Delta_i(k)\|^2\right)^{1/2} \le \sqrt{8}  \sum_{i=1}^n \Delta_i(k).
\end{equation}
By dividing both sides of equation \eqref{error-sum-bound} by $n$, we obtain:
\begin{align*}
    \frac{\sqrt{8}}{n} \left(\sum_{i=1}^n \|\Delta_i(k)\|^2\right)^{1/2} \le \frac{\sqrt{8}}{n}  \sum_{i=1}^n \Delta_i(k) = \sqrt{8} \Delta(k),
\end{align*}

%
The norm of the error due to projection is bounded as:
\begin{equation}\label{projection-error-final}
     \|\bm{\bar\xi}(k)\| \le \sqrt{8} \Delta(k) + \sqrt{2} \frac{\bar L}{n}\alpha.
\end{equation}
By squaring both sides of equation \eqref{projection-error-final}, we obtain
    $\|\bm{\bar\xi}(k)\|^2 \le 8 \|\Delta(k)\|^2 + 2 \bar L^2 \alpha^2(k)$,
which consequently proves Lemma \ref{error-projection}.
\end{proof}


%

\bibliographystyle{IEEEtran}
{\small
\bibliography{mybib1.bib}}

\end{document}